\documentclass[a4paper,12pt,reqno]{amsart}

\usepackage[T1]{fontenc}
\usepackage[utf8]{inputenc}
\usepackage{lmodern}
\usepackage[english]{babel}

\usepackage[%
	left=2.5cm,       
	right=2.5cm,      
	top=3.5cm,        
	bottom=3.5cm,     
	heightrounded,    
	bindingoffset=0mm 
]{geometry}

\usepackage{amsmath}
\usepackage{amssymb}
\usepackage{mathtools}
\usepackage{mathrsfs}
\usepackage[normalem]{ulem}

\usepackage{microtype}

\usepackage[
colorlinks=true,
urlcolor=teal,
citecolor=magenta,
linkcolor=teal,
breaklinks=true]
{hyperref}

\usepackage[noabbrev,capitalize,nameinlink]{cleveref}

\usepackage{bookmark}

\usepackage[initials,msc-links,
]{amsrefs}

\usepackage[dvipsnames]{xcolor}
\usepackage{graphicx}
\usepackage{url}

\usepackage{enumitem}

\numberwithin{equation}{section}

\theoremstyle{plain}
\newtheorem{theorem}{Theorem}[section]
\newtheorem{lemma}[theorem]{Lemma}

\theoremstyle{definition}

\newtheorem{remark}[theorem]{Remark}

\newcommand{\di}{\,\mathrm{d}}
\newcommand{\dive}{\mathrm{div}}
\newcommand{\N}{\mathbb{N}}
\newcommand{\R}{\mathbb{R}}

\DeclarePairedDelimiter{\set}{\{}{\}}

\usepackage[
final
]{showlabels}

\showlabels{bib}

\allowdisplaybreaks
\begin{document}

\title[Failure of the P\'olya--Szeg\H{o} inequality for the fractional gradient]{Failure of the P\'olya--Szeg\H{o} inequality \\ for the fractional gradient}

\author[G.~Stefani]{Giorgio Stefani}
\address[G.~Stefani]{Università degli Studi di Padova, Dipartimento di
Matematica ``Tullio Levi-Civita'', Via Trieste 63, 35121 Padova (PD), Italy}
\email{giorgio.stefani@unipd.it}

\date{\today}




\thanks{\textit{Acknowledgements}.
The author is member of the Istituto Nazionale di Alta Matematica (INdAM),
Gruppo Nazionale per
l’Analisi Matematica, la Probabilità e le loro Applicazioni (GNAMPA), and has
received funding from INdAM under the INdAM--GNAMPA Project 2026 \textit{Metodi non locali classici e distribuzionali per problemi variazionali} (grant agreement No.\ CUP\_E53\-C25\-002\-010\-001), and from the European Union -- NextGenerationEU and the University of Padua under the 2023 STARS@UNIPD  Starting Grant Project \textit{New Directions in Fractional Calculus -- NewFrac} (grant agreement No.\ CUP\_C95\-F21\-009\-990\-001).
This work was partially finalized during the workshop \textit{Nonlocality in Image Denoising and Other Variational Problems}, held at the Department of Mathematics of the University of Pavia from 14 to 17 September 2026. The author would like to thank the organizer, Konstantinos Bessas, and the Department of Mathematics for their hospitality and for providing a stimulating and friendly atmosphere. 
}

\begin{abstract}
We prove that the P\'olya--Szeg\H{o} inequality for the $L^p$ norm of the Riesz fractional $s$-gra\-dient fails for every $p\in[1,2)$ and $s\in(0,1)$, thus providing a negative answer to a question posed by Nguyen and Squassina in 2017.
\end{abstract}

\maketitle

\section{Introduction}

\subsection{Setting}
Let $N\ge1$ and $s\in(0,1)$.
The \emph{Riesz fractional $s$-gradient} of $u\in\operatorname{Lip}_c(\R^N)$ is defined as
\begin{equation}
\label{eq:def_nablas}
\nabla^s u(x)
=
c_{N,s}
\int_{\R^N}\frac{(u(y)-u(x))(y-x)}{|y-x|^{N+s+1}}
\di y,
\quad
x\in\R^N,
\end{equation}
where $c_{N,s}>0$ is a normalization constant whose precise value will play no role.
Analogously, the \emph{Riesz fractional $s$-divergence} of $\varphi\in C^\infty_c(\R^N;\R^N)$ is defined as 
\begin{equation}
\label{eq:def_dives}
\dive^s\varphi(x)
=
c_{N,s}
\int_{\R^N}\frac{(\varphi(y)-\varphi(x))\cdot (y-x)}{|y-x|^{N+s+1}}
\di y,
\quad
x\in\R^N.
\end{equation}  
The operators~\eqref{eq:def_nablas} and~\eqref{eq:def_dives} obey the integration-by-parts formula
\begin{equation}
\label{eq:ibp}
\int_{\R^N}
u\,\dive^s\varphi
\di x
=
-
\int_{\R^N}
\varphi\cdot\nabla^s u\di x.
\end{equation}
We refer to~\cite{silhavy20} for a detailed account of~\eqref{eq:def_nablas}, \eqref{eq:def_dives} and~\eqref{eq:ibp}.
Formula~\eqref{eq:ibp} can be exploited to extend the definition~\eqref{eq:def_nablas} to $L^p$ functions, yielding distributional fractional analogues of  Sobolev and $BV$ spaces.
For a non-exhaustive list of contributions, we refer to~\cites{almi-et-al25a,almi-et-al25b,alicandro-et-al25,antil-et-al24,bellido-et-al20,bellido-et-al21,barrios-medina21,braides-et-al24,brusca-et-al26,caponi-et-al25,carrero-et-al25,duran-maione26,kreisbeck-schonberger22,liu-et-al24,liu-xiao22,lo-rodrigues23,schikorra-et-al17,schonberger24,shieh-spector15,shieh-spector18,spector19,spector20,brue-et-al22,comi-stefani19,comi-et-al22,comi-stefani22,comi-stefani23a,comi-stefani23b,comi-stefani24a,comi-stefani24b,stefani26} and to the references therein.

In~\cite{nguyen-squassina17}*{Open prob.~4.1}, the authors asked whether an analogue of the P\'olya--Szeg\H{o} inequality may hold for the fractional gradient~\eqref{eq:def_nablas} (see also~\cite{carbotti26}*{Sec.~8, point~2}).
Precisely, given $p\in[1,\infty)$, one may ask whether, for every non-negative $u\in C^\infty_c(\R^N)$,
\begin{equation}
\label{eq:polya-szego-s}
\int_{\R^N}|\nabla^s u^\star|^p\di x
\le 
\int_{\R^N}|\nabla^s u|^p\di x,
\end{equation} 
where $u^\star$ is the \emph{symmetric decreasing rearrangement} of $u$,  see~\cite{lieb-loss01}*{Ch.~3} for a detailed presentation.
As observed in~\cite{carbotti26}, inequality~\eqref{eq:polya-szego-s} is true for $p=2$, since
\begin{equation}
\label{eq:equivnorm}
\int_{\R^N}|\nabla^s u|^2\di x
=
\gamma_{N,s}
\iint_{\R^N\times\R^N}
\frac{|u(x)-u(y)|^2}{|x-y|^{N+2s}}
\di x\di y
=
\gamma_{N,s}
[u]_{W^{s,2}}^2
,
\end{equation}
where $\gamma_{N,s}>0$ is a normalization constant,
see~\cite{alicandro-et-al25}*{Prop.~2.8} for a proof.
Hence the case $p=2$ in~\eqref{eq:polya-szego-s} is equivalent to the classical fractional P\'olya--Szeg\H{o} inequality~\cite{almgren-lieb89} for the $W^{s,2}$ seminorm.
For $p\ne2$, instead, the equivalence with the $W^{s,p}$ seminorm  breaks down (see the discussion around~\cite{comi-stefani19}*{Prop.~3.24(a)}) and the validity of~\eqref{eq:polya-szego-s} remains open.

As in the classical setting,  inequality~\eqref{eq:polya-szego-s} would be extremely useful.
In particular, if true for $p=1$, then~\eqref{eq:polya-szego-s} would imply the (currently unknown) isoperimetric property of balls for the \emph{distributional fractional $s$-perimeter} introduced in~\cite{comi-stefani19}*{Sec.~4}; that is, for every $E\subset\R^N$ with $|E|\in(0,\infty)$,
\begin{equation}
\label{eq:iso}
\frac{|D^s\mathbf 1_{B_1}|(\R^N)}{|B_1|^{\frac{N-s}{N}}}
\le 
\frac{|D^s\mathbf 1_{E}|(\R^N)}{|E|^{\frac{N-s}{N}}}
\end{equation}
(see~\cite{comi-stefani19}*{Th.~4.4} for a non-sharp isoperimetric inequality).

\subsection{Main results}
We prove that~\eqref{eq:polya-szego-s} fails for every $p\in[1,2)$ and $s\in(0,1)$.
We first treat exponents close to~$1$ (see also \cref{rem:small_fail} below).

\begin{theorem}[Failure for $p$ close to $1$]
\label{res:small}
Let $s\in(0,1)$.
There exists $\varepsilon_{N,s}>0$ such that $u=\mathbf 1_{B_2}+\mathbf 1_{B_1(\mathrm e_1)}$ and $u^\star=\mathbf 1_{B_2}+\mathbf 1_{B_1}$ satisfy 
\begin{equation}
\label{eq:small}
\|\nabla^s u^\star\|_{L^p}
>
\|\nabla^s u\|_{L^p}
\quad
\text{for every}\ p\in[1,1+\varepsilon_{N,s}].
\end{equation}
\end{theorem}

The proof of \cref{res:small} is elementary and relies on the properties of the fractional variation of balls established in~\cites{comi-et-al22,comi-stefani24b}. The underlying idea is to superimpose a translated characteristic function on a larger ball and to observe that, at $p=1$, rearrangement aligns the corresponding fractional gradients and strictly increases the energy. Stability with respect to $p$ then yields the failure of~\eqref{eq:polya-szego-s} for $p>1$ sufficiently close to~$1$.
It is worth observing that \cref{res:small} does \emph{not} disprove~\eqref{eq:iso}. 

Since~\eqref{eq:polya-szego-s} holds for $p=2$ by~\eqref{eq:equivnorm}, we must let the counterexample depend on~$p$  as $p\to2^-$.
This leads to our second main result.

\begin{theorem}[Failure for every $p\in(1,2)$]
\label{res:big}
Let $s\in(0,1)$.
For every $p\in(1,2)$, there exists $u\in C^\infty_c(\R^N)$ such that $u^\star\in C^\infty_c(\R^N)$ and 
\begin{equation*}
\|\nabla^s u^\star\|_{L^p}
>
\|\nabla^s u\|_{L^p}.
\end{equation*}
\end{theorem}

The proof of \cref{res:big} relies on the same perturbative mechanism underlying \cref{res:small}, but requires a more flexible construction. 
We start from a smooth radially symmetric decreasing function which is constant near the origin and perturb its plateau by a small translated bump. 
The fractional gradient near the center produces a singular leading term in the first variation, while the remaining contributions are of lower order. By first choosing the scale of the bump and then its amplitude, the leading term dominates the nonlinear remainder and yields a counterexample for every $p\in(1,2)$.

\section{Proofs of the results}

\subsection{Proof of \texorpdfstring{\cref{res:small}}{Theorem 1.1}}

In the proof of \cref{res:small} we exploit the following well-known facts.
Given $s\in(0,1)$, by~\cite{comi-et-al22}*{Cor.~1} and~\cite{comi-stefani24b}*{Prop.~1.15}, for each $c\in\R^N$ and $r>0$, we have that
\begin{equation}
\label{eq:contleb}
\nabla^s\mathbf 1_{B_r(c)}
\in 
C(\R^N\setminus\partial B_r(c))
\cap L^q(\R^N)
\quad
\text{for every}\ q\in\big[1,\tfrac1s\big),
\end{equation}
\begin{equation}
\label{eq:pos}
|\nabla^s\mathbf 1_{B_r(c)}(x)|
>0
\quad
\text{for every}\ x\in\R^N\setminus(\partial B_r(c)\cup\set*{c}),
\end{equation}
\begin{equation}
\label{eq:normal}
\nabla^s\mathbf 1_{B_r(c)}(x)
=
-\frac{x-c}{|x-c|}
\,
|\nabla^s\mathbf 1_{B_r(c)}(x)|
\quad
\text{for every}\ x\in\R^N\setminus(\partial B_r(c)\cup\set*{c})
.
\end{equation}
In addition, by~\cite{comi-et-al22}*{Th.~8}, if $u\in BV(\R^N)\cap L^\infty(\R^N)$, then $\nabla^s u\in L^p(\R^N)$ for every $p\in[1,1/s)$
and $s\in(0,1)$ and thus, by the Dominated Convergence Theorem,
\begin{equation}
\label{eq:pstab}
\lim_{p\to1^+}
\|\nabla^s u\|_{L^p}
=
\|\nabla^s u\|_{L^1}.
\end{equation} 

\begin{proof}[Proof of \cref{res:small}]
Since $B_1(\mathrm e_1)\subset B_2$, we have $u^\star=\mathbf 1_{B_2}+\mathbf 1_{B_1}$.
Owing to~\eqref{eq:normal},
\begin{equation*}
\begin{split}
\|\nabla^s u^\star\|_{L^1}
&=
\int_{\R^N}
|\nabla^s\mathbf 1_{B_2}(x)+\nabla^s\mathbf 1_{B_1}(x)|
\di x
=
\int_{\R^N}
\bigg|
\frac{x}{|x|}
\,
\Big(|\nabla^s\mathbf 1_{B_2}(x)|
+
|\nabla^s\mathbf 1_{B_1}(x)|
\Big)
\bigg|
\di x
\\
&=
\int_{\R^N}
|\nabla^s\mathbf 1_{B_2}(x)|
+
|\nabla^s\mathbf 1_{B_1}(x)|
\di x
=
\|\nabla^s\mathbf 1_{B_2}\|_{L^1}
+
\|\nabla^s\mathbf 1_{B_1}\|_{L^1}.
\end{split}
\end{equation*}
Hence, by translation invariance of the fractional gradient~\eqref{eq:def_nablas},
\begin{equation*}
\|\nabla^s u^\star\|_{L^1}
=
\|\nabla^s\mathbf 1_{B_2}\|_{L^1}
+
\|\nabla^s\mathbf 1_{B_1}\|_{L^1}
=
\|\nabla^s\mathbf 1_{B_2}\|_{L^1}
+
\|\nabla^s\mathbf 1_{B_1(\mathrm e_1)}\|_{L^1}.
\end{equation*}
As a consequence, by the triangular inequality,
\begin{equation*}
\begin{split}
\|\nabla^s u^\star\|_{L^1}
\ge 
\|\nabla^s\mathbf 1_{B_2}
+
\nabla^s\mathbf 1_{B_1(\mathrm e_1)}\|_{L^1}
=
\|\nabla^s u\|_{L^1}.
\end{split}
\end{equation*}
We claim that the above inequality is strict.
Indeed, by~\eqref{eq:normal}, we have that
\begin{equation*}
\nabla^s \mathbf 1_{B_2}(\mathrm e_1/2)
=
-\mathrm e_1
\,
|\nabla^s\mathbf 1_{B_2}(\mathrm e_1/2)|
\quad
\text{and}
\quad
\nabla^s\mathbf 1_{B_1(\mathrm e_1)}(\mathrm e_1/2)
=
\mathrm e_1
\,
|\nabla^s\mathbf 1_{B_1(\mathrm e_1)}(\mathrm e_1/2)|.
\end{equation*}
Hence, owing to~\eqref{eq:pos},
\begin{equation*}
\begin{split}
|\nabla^s\mathbf 1_{B_2}(\mathrm e_1/2)|
+
|\nabla^s\mathbf 1_{B_1(\mathrm e_1)}(\mathrm e_1/2)|
&>
\big|
|\nabla^s\mathbf 1_{B_2}(\mathrm e_1/2)|
-
|\nabla^s\mathbf 1_{B_1(\mathrm e_1)}(\mathrm e_1/2)|
\big|
\\
&=
|\nabla^s\mathbf 1_{B_2}(\mathrm e_1/2)
-
\nabla^s\mathbf 1_{B_1(\mathrm e_1)}(\mathrm e_1/2)|.
\end{split}
\end{equation*}
Thus, by the continuity in~\eqref{eq:contleb}, for some open neighborhood $U\subset\R^N$ of $\mathrm e_1/2$,
\begin{equation*}
|\nabla^s\mathbf 1_{B_2}(x)|
+
|\nabla^s\mathbf 1_{B_1(\mathrm e_1)}(x)|
>
|\nabla^s\mathbf 1_{B_2}(x)
+
\nabla^s\mathbf 1_{B_1(\mathrm e_1)}(x)|
\quad
\text{for every}\ 
x\in U.
\end{equation*}
Integrating the previous strict inequality yields
\begin{equation}
\label{eq:uno}
\|\nabla^s u^\star\|_{L^1}
>
\|\nabla^s u\|_{L^1}.
\end{equation}
The conclusion hence follows by combining~\eqref{eq:uno} with~\eqref{eq:pstab}.
\end{proof}

\begin{remark}
\label{rem:small_fail}
Let $L^{s,p}(\R^N)$ be the completion of $C^\infty_c(\R^N)$ with respect to the norm $u\mapsto\|u\|_{L^p}+\|\nabla^s u\|_{L^p}$. 
For every $s\in(0,1)$ and $p\in(1,\infty)$, $L^{s,p}(\R^N)$ coincides with the \emph{Bessel potential space}, see~\cites{comi-stefani19,brue-et-al22,kreisbeck-schonberger22} (see~\cite{comi-stefani19}*{Th.~3.23} for $p=1$).
Since $u,u^\star\in BV(\R^N)\cap L^\infty(\R^N)$ in \cref{res:small} have compact support, a plain convolution argument proves that $u,u^\star\in L^{s,p}(\R^N)$ for every $s\in(0,1)$ and $p\in[1,2)$ such that $sp<1$.
Hence, for each $s\in(0,1)$ and $p\in[1,2)$ such that $sp<1$, there is $(v_k)_{k\in\N}\subset C^\infty_c(\R^N)$ such that
\begin{equation*}
\lim_{k\to\infty}\|\nabla^s v_k\|_{L^p}
=
\|\nabla^s u\|_{L^p}
\quad
\text{and}
\quad
\|\nabla^s u^\star\|_{L^p}
\le
\liminf_{k\to\infty}\|\nabla^s v_k^\star\|_{L^p}.
\end{equation*}
This proves that, for every $s\in(0,1)$, there exists $\eta_{N,s}>0$ with the following property: for each $p\in[1,1+\eta_{N,s}]$, there exists $v\in C^\infty_c(\R^N)$ such that 
$\|\nabla^sv\|_{L^p}<\|\nabla^sv^\star\|_{L^p}$.
This disproves~\eqref{eq:polya-szego-s} for every $p\in[1,2)$ sufficiently close to~$1$.
\end{remark}

\subsection{Proof of \texorpdfstring{\cref{res:big}}{Theorem 1.2}}

We begin with three preliminary lemmas. 

\begin{lemma}\label{lem:local-gradient}
Let $f\in C^\infty_c(\R^N)$ be a radially symmetric decreasing function such that
$0\leq f\leq1$, $f=1\ \text{on}\ B_2$ and $
\operatorname{supp}f\subset B_3$.
Then, there exists $b>0$ such that
\begin{equation}
\label{eq:local-nabla}
\nabla^s f(x)=-bx+O(|x|^3)
\quad
\text{as }x\to0.
\end{equation}
Consequently, letting $d=(N+p-2)b^{p-1}>0$, there exist an open neighborhood $U$ of the origin and $g\in C^1(U)$ such that $g(x)=O(|x|^p)$, $\nabla g(x)=O(|x|^{p-1})$ as $x\to0$ and 
\begin{equation}\label{eq:local-divA}
-\dive\big(|\nabla^sf|^{p-2}\,\nabla^sf\big)(x)
=
d|x|^{p-2}+g(x)
\end{equation}
pointwise for every $x\in U\setminus\set*{0}$ and in the distributional sense on $U$.
\end{lemma}

\begin{proof}
We divide the proof in two parts.

\vspace{1ex}

\textit{Part 1: proof of~\eqref{eq:local-nabla}}.
Let us set $F(x)=\nabla^s f(x)$ for every $x\in\R^N$ for brevity.
Since $f=1$ in $B_2$, we have $F\in C^\infty(B_1)$ with $F(0)=0$, because $f$ is radially symmetric.
Moreover, by differentiating under the integral sign, 
for each $i,j\in\set*{1,\dots,N}$, we get
\begin{equation*}
\partial_j
F_i(0)
=
c_{N,s}
\int_{\R^N}
(1-f(y))
\,
\bigg(
\frac{\delta_{ij}}{|y|^{N+s+1}}
-
(N+s+1)\,\frac{y_i\,y_j}{|y|^{N+s+3}}
\bigg)
\di y.
\end{equation*}
Again by the radial symmetry of $f$,
\begin{equation*}
\int_{\R^N}
(1-f(y))
\,
\frac{y_i\, y_j}{|y|^{N+s+3}}
\di y
=
\frac{\delta_{ij}}{N}
\int_{\R^N}
\frac{1-f(y)}{|y|^{N+s+1}}
\di y,
\end{equation*}
and thus $DF(0)=-b\,\mathrm{Id}$, where 
\begin{equation*}
b
=
c_{N,s}\,\frac{1+s}{N}
\int_{\R^N}
\frac{1-f(y)}{|y|^{N+s+1}}
\di y
>0
.
\end{equation*}
In addition, since $F$ is odd by definition, we also have that $D^2F(0)=0$. 
As a consequence, 
\begin{equation*}
F(x)=\nabla^s f(x)=-bx+O(|x|^3)
\quad
\text{as}\ x\to0,
\end{equation*} 
concluding the proof of~\eqref{eq:local-nabla}.

\vspace{1ex}

\textit{Part 2: proof of~\eqref{eq:local-divA}}. 
Again by the radial symmetry of~$f$, for every $\mathcal R\in\mathrm{O}(N)$, we have that $F(\mathcal R(x))=\mathcal R(F(x))$ 
for every $x\in B_1$.
As a consequence, there exists a smooth function $\lambda\colon[0,1]\to[0,\infty)$ such that 
\begin{equation}
\label{eq:radod}
F(x)=-\lambda(|x|)\,x
\quad
\text{for every}\ x\in B_1.
\end{equation}
We also note that $\lambda(0)=b$ and that $\lambda'(0)=0$, because the function $t\mapsto F(t\mathrm e_1)=\lambda(|t|)\,t\mathrm e_1$ is smooth (and odd) for every $t\in(-1,1)$.
Now let $A(x)=|F(x)|^{p-2}\,F(x)$ for every $x\in\R^N$ for brevity, with the convention that $A(x)=0$ for every $x\in\R^N$ such that $F(x)=0$.
Being $p>1$ by assumption, we have that $A\in C_b(\R^N)$.
By~\eqref{eq:radod}, we can rewrite $A(x)=-\lambda(|x|)^{p-1}|x|^{p-2}\,x$ for every $x\in B_1$. 
Letting $w(r)=\lambda(r)^{p-1}$ for every $r\in[0,1]$,
we hence get that, for every $x\in B_1\setminus\set*{0}$, 
\begin{equation*}
-
\dive(A(x))
=
\dive\big(\lambda(|x|)^{p-1}|x|^{p-2}\,x\big)
=
(N+p-2)\,w(|x|)|x|^{p-2}
+
w'(|x|)\,|x|^{p-1}.
\end{equation*}
As a consequence, we thus get that 
\begin{equation}
\label{eq:diva}
-\dive(A(x))
=
d|x|^{p-2}
+
g(x)
\quad
\text{for every}\ x\in B_1\setminus\set*{0},
\end{equation}
where $d=(N+p-2)b^{p-1}>0$ and $g\colon B_1\setminus\set*{0}\to\R$ is defined, for every $x\in B_1\setminus\set*{0}$,  as
\begin{equation}
\label{eq:restoq}
g(x)
=
(N+p-2)
\big(
w(|x|)-b^{p-1}
\big)
\,|x|^{p-2}
+
w'(|x|)\,|x|^{p-1}.
\end{equation}
Since  $w(0)=\lambda(0)^{p-1}=b^{p-1}$ and 
$w'(0)
=(p-1)\,b^{p-2}\,\lambda'(0)=0$, we have 
\begin{equation}
\label{eq:wzero}
w(r)=b^{p-1}+O(r^2)
\quad
\text{as}\ r\to0^+.
\end{equation}
In addition, being $\lambda'(0)=0$, we have $\lambda'(r)=O(r)$ as $r\to0^+$ and thus 
\begin{equation}
\label{eq:wuno}
w'(r)=(p-1)\,\lambda(r)^{p-2}\,\lambda'(r)
=
O(\lambda'(r))
=
O(r)
\quad
\text{as}\ r\to0^+.
\end{equation}
Similarly, we can estimate
\begin{equation}
\label{eq:wdue}
w''(r)
=
(p-1)(p-2)\,\lambda(r)^{p-3}\,\lambda'(r)^2
+
(p-1)\,\lambda(r)^{p-2}\,\lambda''(r)
=
O(1)
\quad
\text{as}\ r\to0^+.
\end{equation}
As a consequence, by exploiting~\eqref{eq:wzero}, \eqref{eq:wuno} and~\eqref{eq:wdue} in~\eqref{eq:restoq}, we infer that 
\begin{equation*}
|g(x)|
=
O(|x|^p)
\quad
\text{and}
\quad
\nabla g(x)
=
O(|x|^{p-1})
\quad
\text{as}\ x\to0.
\end{equation*}
Therefore, defining $g(0)=0$, we obtain that $g\in C^1(U)$ for some open neighborhood $U\subset\R^N$ of the origin.
This concludes the proof of the validity of~\eqref{eq:local-divA} in the pointwise sense on $U\setminus\set*{0}$.
To prove the validity of~\eqref{eq:local-divA} in the weak sense on $U$, it is enough to observe that $|A(x)|=O(|x|^{p-1})$ for $x\in B_1$ by definition and that, given $\varphi\in C^\infty_c(\R^N)$ with $\operatorname{supp}\varphi\subset B_\rho$ for some $B_\rho\Subset U$, we can estimate
\begin{equation*}
\bigg|
\int_{B_\rho\setminus B_\varepsilon}
A\cdot\nabla\varphi\di x
+
\int_{B_\rho\setminus B_\varepsilon}
\varphi\,\dive A
\di x
\bigg|
=
\bigg|
\int_{\partial B_\varepsilon}
\varphi\,A\cdot\nu_{B_\varepsilon}\di\mathscr{H}^{N-1}
\bigg|
\le 
C\|\varphi\|_{L^\infty}
\varepsilon^{N+p-2}
\end{equation*}
for every $\varepsilon\in(0,\rho)$, which vanishes as $\varepsilon\to0^+$. This concludes the proof.
\end{proof}

\begin{lemma}\label{lem:local-first-variation}
With $f\in C^\infty_c(\R^N)$ and $d>0$ as in \cref{lem:local-gradient}, there exist $\delta>0$ and $H\in C^1(B_{\delta/2})$ such that
\begin{equation}\label{eq:local-first-variation}
\int_{\R^N}
|\nabla^sf|^{p-2}\,\nabla^sf\cdot\nabla^s\varphi
\di x
=
\int_{\R^N}
\bigl(dV_\delta+H\bigr)\varphi
\di x
\end{equation}
for every $\varphi\in C^\infty_c(B_{\delta/2})$, where
$
V_\delta
=
I_{1-s}
\bigl(
|\cdot|^{p-2}\,\mathbf1_{B_\delta}
\bigr).
$
\end{lemma}

\begin{proof}
We keep the same notation of the proof of \cref{lem:local-gradient} and set $F=\nabla^s f\in C^\infty_b(\R^N)$ and $A=|F|^{p-2}F\in C_b(\R^N)$. 
Let $\delta>0$ be sufficiently small that~\eqref{eq:local-divA} holds with 
$U=B_{4\delta}$ and fix a cut-off $\chi\in C^\infty_c(B_{3\delta})$ such that $0\le\chi\le1$ and $\chi=1$ on $B_{2\delta}$.
We decompose $A=A_0+A_1$, where $A_0=\chi A$ and $A_1=(1-\chi)A$.
On the one side, since  $A_0\in W^{1,1}(\R^N;\R^N)$ has compact support, for every
$\varphi\in C^\infty_c(B_{\delta/2})$ we have that
\begin{equation}
\label{eq:banana}
\int_{\R^N}A_0\cdot\nabla^s\varphi\di x
=
\int_{\R^N}A_0\cdot\nabla(I_{1-s}\varphi)\di x
=
-
\int_{\R^N}
I_{1-s}\varphi
\,
\dive A_0
\di x.
\end{equation}
Here we have exploited the representation $\nabla^s=I_{1-s}\nabla$ (see~\cite{comi-stefani19}*{Prop.~2.2} for instance), where $I_\alpha$ is the \emph{Riesz potential} of order $\alpha\in(0,N)$.  
By~\eqref{eq:local-divA} in \cref{lem:local-gradient}, we have 
\begin{equation*}
-\dive A_0(x)
=
d\chi(x)|x|^{p-2}
+
G(x)
\quad
\text{for every}\ x\in\R^N,
\end{equation*}
where we have set $G=\chi\,g-\nabla\chi\cdot A\in C^1_c(\R^N)$.
Hence, recalling~\eqref{eq:banana}, we get 
\begin{equation}
\label{eq:acca1}
\begin{split}
\int_{\R^N}A_0
&
\cdot\nabla^s\varphi\di x
=
d
\int_{\R^N}
\varphi
\,
I_{1-s}\big(|\cdot|^{p-2}\chi\big)
\di x
+
\int_{\R^N}
\varphi\,I_{1-s}G\di x
\\
&=
d
\int_{\R^N}
\varphi
\,
I_{1-s}\big(|\cdot|^{p-2}\mathbf 1_{B_\delta}\big)
\di x
+
\int_{\R^N}
\varphi
\,
I_{1-s}\big(|\cdot|^{p-2}\,d\,(\chi-\mathbf 1_{B_\delta})+G\big)
\di x
\\
&=
\int_{\R^N}
\big(dV_\delta+H_0\big)
\varphi
\di x
\end{split}
\end{equation}
for every
$\varphi\in C^\infty_c(B_{\delta/2})$,
where we have set $H_0=I_{1-s}\big(|\cdot|^{p-2}\,d\,(\chi-\mathbf 1_{B_\delta})+G\big)\in C^1(B_{\delta/2})$.
On the other side, given $\varphi\in C^\infty_c(B_{\delta/2})$, for every $x\in\R^N$ we can write 
\begin{equation*}
\nabla^s\varphi(x)
=
c_{N,s}
\lim_{\varepsilon\to0^+}
\int_{B_\varepsilon(x)^c}
\frac{\varphi(y)\,
(y-x)}{|y-x|^{N+s+1}}
\di y
=
c_{N,s}
\lim_{\varepsilon\to0^+}
\int_{B_{\delta/2}\setminus B_\varepsilon(x)}\frac{\varphi(y)
\,
(y-x)}{|y-x|^{N+s+1}}\di y,
\end{equation*} 
where $c_{N,s}>0$ is the constant appearing in~\eqref{eq:def_nablas} (see~\cite{comi-stefani19}*{Sec.~2.2}).
Hence, since $A_1\in L^\infty(\R^N;\R^N)$,  $A_1=0$ on $B_{2\delta}$, and 
\begin{equation}
\label{eq:distanti}
|x-y|\ge\frac32\delta
\quad
\text{whenever}\
x\in B_{2\delta}^c\ 
\text{and}\
y\in B_{\delta/2},
\end{equation}
owing to Tonelli's Theorem we can estimate
\begin{equation*}
\begin{split}
\bigg|
\int_{\R^N}A_1(x)\cdot\nabla^s\varphi\di x
\bigg|
&\le 
c_{N,s}
\int_{B_{2\delta}^c}
|A_1(x)|
\int_{B_{\delta/2}}
\frac{|\varphi(y)|}{|y-x|^{N+s}}
\di y
\di x
\\
&\le 
c_{N,s}
\,
\|A_1\|_{L^\infty}
\,
\|\varphi\|_{L^1(B_{\delta/2})}
\int_{B_{3\delta/2}^c}\frac{\,\di h}{|h|^{N+s}}
<\infty.
\end{split}
\end{equation*}
Hence, using~\eqref{eq:distanti} and Fubini's Theorem, we obtain that
\begin{equation}
\label{eq:acca2}
\begin{split}
\int_{\R^N}A_1\cdot\nabla^s\varphi\di x 
&=
c_{N,s}
\int_{B_{2\delta}^c}
A_1(x)
\cdot
\int_{B_{\delta/2}}
\frac{\varphi(y)
\,(y-x)}{|y-x|^{N+s+1}}\di y
\di x
\\
&=
c_{N,s}
\int_{B_{\delta/2}}
\varphi(y)
\int_{B_{2\delta}^c}
\frac{A_1(x)
\cdot(y-x)}{|y-x|^{N+s+1}}
\di x
\di y
=
\int_{\R^N}\varphi\,H_1\di x
\end{split}
\end{equation}
where, for every $x\in B_{\delta/2}$, we have set
\begin{equation*}
H_1(x)
=
c_{N,s}
\int_{B_{2\delta}^c}
\frac{A_1(y)
\cdot(x-y)}{|x-y|^{N+s+1}}
\di y.
\end{equation*}
Again thanks to~\eqref{eq:distanti}, differentiating under the integral sign readily gives $H_1\in C^1(B_{\delta/2})$.
By combining~\eqref{eq:acca1} and~\eqref{eq:acca2} and defining $H=H_0+H_1\in C^1(B_{\delta/2})$, we get~\eqref{eq:local-first-variation}.
\end{proof}

\begin{lemma}\label{lem:short-translation}
Let $\delta>0$ and $V_\delta$ be as in \cref{lem:local-first-variation}.
Assume that $\psi\in C^\infty_c(B_1)$ is a non-negative and radially symmetric decreasing function such that $\int_{\R^N}\psi\,\di x=1$.
Then, for every $e\in \mathbb S^{N-1}$, there exists $c>0$ such that
\begin{equation}\label{eq:VR-translation}
(V_\delta*\psi_r)(0)
-
(V_\delta*\psi_r)(re)
\ge
cr^{p-1-s}
\quad
\text{for every}\
r\in\big(0,\tfrac\delta2\big), 
\end{equation}
where $\psi_r(x)=r^{-N}\psi(x/r)$ for every $x\in\R^N$ and $r>0$.
\end{lemma}

\begin{proof}
By the layer-cake identity, for every $y\in\R^N$, we can write
\begin{equation*}
|y|^{p-2}
\,
\mathbf1_{B_\delta}(y)
=
\delta^{p-2}
\,
\mathbf1_{B_\delta}(y)
+
(2-p)
\int_0^\delta
t^{p-3}
\,
\mathbf1_{B_t}(y)\di t
\end{equation*}
Hence, given $e\in\mathbb S^{N-1}$ and $r>0$, by Tonelli's Theorem we get that
\begin{equation*}
V_\delta
=
I_{1-s}\big(|\cdot|^{p-2}\,\mathbf1_{B_\delta}\big)
=
\delta^{p-2}
I_{1-s}\mathbf1_{B_\delta}
+
(2-p)
\int_0^\delta
t^{p-3}
I_{1-s}\mathbf 1_{B_t}\di t.
\end{equation*}
Setting $P_t=I_{1-s}\mathbf1_{B_t}$ for every $t>0$ for brevity, the above identity rewrites as
\begin{equation}
\label{eq:combo}
V_\delta
=
\delta^{p-2}
P_\delta
+
(2-p)
\int_0^\delta
t^{p-3}
P_t
\di t
.
\end{equation}
Now, since $P_t$ is radially strictly decreasing for every $t>0$ and $\psi_r$ is non-trivial, non-negative and radially symmetric decreasing, $P_t*\psi_r$ is radially strictly decreasing for each $t>0$.
Thus, by~\eqref{eq:combo} and Tonelli's Theorem again, and assuming $2r<\delta$, we can estimate
\begin{equation*}
\begin{split}
(V_\delta*\psi_r)(0)
-
(V_\delta*\psi_r)(re)
&=
\delta^{p-2}
\big(
P_\delta*\psi_r(0)-P_\delta*\psi_r(re)
\big)
\\
&\quad+
(2-p)
\int_0^\delta
t^{p-3}
\big(P_t*\psi_r(0)-P_t*\psi_r(re)\big)
\di t
\\
&\ge
(2-p)
\int_r^{2r}
t^{p-3}
\big(P_t*\psi_r(0)-P_t*\psi_r(re)\big)
\di t.
\end{split}
\end{equation*}
By changing variables, for every $t,r>0$ and $x\in\R^N$ we have that  
\begin{equation*}
(P_{t}*\psi_r)(x)
=
r^{1-s}(P_{t/r}*\psi)(x/r).
\end{equation*}
Therefore, a change of variable yields that 
\begin{equation*}
\int_r^{2r}
t^{p-3}
\big(P_t*\psi_r(0)-P_t*\psi_r(re)\big)
\di t
=
r^{p-1-s}
\int_1^2
\tau^{p-3}
\big(P_\tau*\psi(0)-P_\tau*\psi(e)\big)
\di\tau
\end{equation*}
and thus, for every $r<\frac\delta2$, 
\begin{equation*}
(V_\delta*\psi_r)(0)
-
(V_\delta*\psi_r)(re)
\ge
cr^{p-1-s}, 
\end{equation*}
where 
\begin{equation*}
c
=
(2-p)
\int_1^2
\tau^{p-3}
\big(P_\tau*\psi(0)-P_\tau*\psi(e)\big)
\di\tau>0,
\end{equation*}
proving~\eqref{eq:VR-translation} and concluding the proof.
\end{proof}

\begin{proof}[Proof of \cref{res:big}]
Fix $p\in(1,2)$ and let $f\in C^\infty_c(\R^N)$ be as in \cref{lem:local-gradient}, $\delta>0$ and $H\in C^1(B_{\delta/2})$ be as in \cref{lem:local-first-variation}, and $\psi\in C^\infty_c(\R^N)$ be as in \cref{lem:short-translation}.
Fix $e\in\mathbb S^{N-1}$ and, for every $r>0$, define
\begin{equation}
\label{eq:def_mer}
M_e(r)
=
\int_{\R^N}
|\nabla^sf|^{p-2}\,\nabla^s f\cdot
\nabla^s
\bigl(
\psi_r-\psi_r(\cdot-re)
\bigr)
\di x.
\end{equation}
Since $\psi$ has compact support, both functions $\psi_r$ and $\psi_r(\cdot-re)$ are supported in $B_{\delta/2}$ for every $r>0$ sufficiently small.
By \cref{lem:local-first-variation,lem:short-translation}, we hence infer that 
\begin{equation}
\label{eq:mer_cc}
\begin{split}
M_e(r)
&=
d\,\Big(
(V_\delta*\psi_r)(0)
-
(V_\delta*\psi_r)(re)
\Big)
+
\int_{\R^N}
\psi(z)
\bigl[
H(rz)-H(rz+re)
\bigr]
\di z
\\
&\geq
c r^{p-1-s}-Cr
\end{split}
\end{equation}
for every $r>0$ sufficiently small, for some constants $c,C>0$ independent of $r$.
Since $p-1-s<1$, we thus get that
\begin{equation}
\label{eq:mer_est}
M_e(r)\ge\widetilde cr^{p-1-s}
\end{equation}
for every $r>0$ sufficiently small, where $\widetilde c>0$ depends on $c$ and $C$ in~\eqref{eq:mer_cc} only.
For such $r>0$, we define the families $(u_\varepsilon)_{\varepsilon>0},(v_\varepsilon)_{\varepsilon>0}\subset C^\infty_c(\R^N)$ by letting  $u_\varepsilon=f+\varepsilon\psi_r(\cdot-re)$ and $v_\varepsilon=f+\varepsilon\psi_r$ for every $\varepsilon>0$.
Up to taking $r>0$ smaller if necessary, for every $\varepsilon>0$ both $u_\varepsilon$ and $v_\varepsilon$ are supported in~$B_3$.
Moreover, for every $\varepsilon>0$, we have  $u_\varepsilon,v_\varepsilon\ge0$ and 
\begin{equation*}
|\set*{u_\varepsilon>\lambda}|
=
|\set*{v_\varepsilon>\lambda}|
=
\begin{cases}
|\set*{f>\lambda}| 
& \text{for}\ \lambda\in(0,1)
\\[1ex]
|\set*{\varepsilon\psi_r>\lambda-1}|
& \text{for}\ \lambda\ge1.
\end{cases}
\end{equation*}
Therefore, since $v_\varepsilon$ is radially symmetric decreasing because  both~$f$ and~$\psi$ are by the assumptions in \cref{lem:local-gradient,lem:short-translation}, we infer that $u_\varepsilon^\star=v_\varepsilon$ for every $\varepsilon>0$.
By applying~\eqref{eq:ele} in \cref{res:ele} below with $\xi=\nabla^s f$ and $\eta=\varepsilon\nabla^s\psi_r$, we get 
\begin{equation*}
\begin{split}
\|\nabla^s u_\varepsilon^\star\|^p_{L^p}
&=
\int_{\R^N}
|\nabla^s v_\varepsilon|^p
\di x
\ge 
\int_{\R^N}
|\nabla^sf|^p
\di x
+
p\varepsilon\int_{\R^N}
|\nabla^s f(x)|^{p-2}
\,
\nabla^s f(x)\cdot\nabla^s\psi_r(x)
\di x
,
\end{split}
\end{equation*}
while, by applying~\eqref{eq:ele} below with $\xi=\nabla^s f$ and $\eta=\varepsilon\nabla^s\psi_r(\cdot-re)$, we obtain 
\begin{equation*}
\begin{split}
\|\nabla^s u_\varepsilon\|^p_{L^p}
&= 
\int_{\R^N}
|\nabla^s f(x)+\varepsilon\nabla^s\psi_r(x-re)|^p
\di x
\\
&\le 
\int_{\R^N}
|\nabla^s f(x)|^p
\di x
+
C_p\varepsilon^p
\int_{\R^N}
|\nabla^s\psi_r(x-re)|^p
\di x
\\
&\quad
+
p\varepsilon\int_{\R^N}
|\nabla^s f(x)|^{p-2}
\,
\nabla^s f(x)\cdot\nabla^s\psi_r(x-re)
\di x.
\end{split}
\end{equation*}
As a consequence, recalling~\eqref{eq:def_mer},  we can estimate
\begin{equation*}
\begin{split}
\|\nabla^s u_\varepsilon^\star\|^p_{L^p}
-
\|\nabla^s u_\varepsilon\|^p_{L^p}
&\ge
p\varepsilon\int_{\R^N}
|\nabla^s f(x)|^{p-2}
\,
\nabla^s f(x)\cdot\big(\nabla^s\psi_r(x)-\nabla^s\psi_r(x-re)\big)
\di x
\\
&\quad-
C_p\varepsilon^p
\int_{\R^N}
|\nabla^s\psi_r(x-re)|^p
\di x
\\
&=
p\varepsilon M_e(r)
-
C\varepsilon^p
\|\nabla^s\psi_r\|_{L^p}^p
.
\end{split}
\end{equation*}
Observing that, by the homogeneity of the fractional gradient, 
\begin{equation*}
\|\nabla^s\psi_r\|_{L^p}^p
=
r^{-N(p-1)-sp}
\,
\|\nabla^s\psi\|_{L^p}^p
\end{equation*}
and recalling~\eqref{eq:mer_est}, we get that, for every $r>0$ sufficiently small and for every $\varepsilon>0$,
\begin{equation*}
\|\nabla^su_\varepsilon^\star\|^p_{L^p}
-
\|\nabla^s u_\varepsilon\|^p_{L^p}
\ge 
\widetilde c\varepsilon r^{p-1-s}
-
\widetilde C\varepsilon^p
r^{-N(p-1)-sp},
\end{equation*}
where $\widetilde c>0$ is as in~\eqref{eq:mer_est} and $\widetilde C>0$ is independent of $r$ and $\varepsilon$. 
Hence fixing $r>0$ and choosing $\varepsilon>0$ small enough to ensure that $\widetilde C\varepsilon^{p-1}<\widetilde c r^{p-1-s+N(p-1)+sp}$ gives that $\|\nabla^su_\varepsilon^\star\|^p_{L^p}
-
\|\nabla^s u_\varepsilon\|^p_{L^p}>0$, concluding the proof.
\end{proof}

In the proof of \cref{res:big} we exploited the following elementary inequality.
We provide a brief proof of it for the reader's convenience.

\begin{lemma}
\label{res:ele}
For every $p\in(1,2)$, there exists $C_p>0$ such that
\begin{equation}
\label{eq:ele}
0\le |\xi+\eta|^p-|\xi|^p-p|\xi|^{p-2}\xi\cdot\eta
\le C_p|\eta|^p
\quad
\text{for every}\ \xi,\eta\in\R^N.
\end{equation}
\end{lemma}

\begin{proof}
The lower bound follows from the convexity of $z\mapsto |z|^p$. 
For the upper bound, by the Fundamental Theorem of Calculus, we can estimate
\begin{equation*}
\begin{split}
|\xi+\eta|^p-|\xi|^p-p|\xi|^{p-2}\xi\cdot\eta
&
=
p\int_0^1
\big(|\xi+t\eta|^{p-2}(\xi+t\eta)-|\xi|^{p-2}\xi\big)
\cdot\eta
\di t
\\
&
\le
C_p|\eta|^p\int_0^1t^{p-1}\,dt
\le C_p|\eta|^p
\end{split}
\end{equation*}
where we used the elementary inequality
\begin{equation}
\label{eq:ele2}
\big||a|^{p-2}a-|b|^{p-2}b\big|
\le C_p|a-b|^{p-1},
\quad 
\text{for every}\ 
a,b\in\R^N.
\end{equation}
To prove~\eqref{eq:ele2}, we distinguish two cases.
If
$|a-b|\ge \frac12\max\set*{|a|,|b|}$,
then
\begin{equation*}
||a|^{p-2}a-|b|^{p-2}b|
\le |a|^{p-1}+|b|^{p-1}
\le 
2^p|a-b|^{p-1}.
\end{equation*}
If instead $|a-b|<\frac12\max\set*{|a|,|b|}$ then we may suppose that $|a|\ge |b|$ without loss of generality and thus $
|a-b|<\frac12|a|$.
Hence, for every $t\in[0,1]$, we have
\begin{equation*}
|a+t(b-a)|
\ge |a|-t|a-b|
\ge |a|-|a-b|
> \frac12|a|.
\end{equation*}
Thus, again by the Fundamental Theorem of Calculus applied to $z\mapsto|z|^{p-2}z$, we get
\begin{equation*}
\big||a|^{p-2}a-|b|^{p-2}b\big|
\le 
C_p|a-b|
\int_0^1 |a+t(b-a)|^{p-2}
\di t
\le 
C_p|a|^{p-2}|a-b|,
\end{equation*}
since $
\|\nabla(|z|^{p-2}z)\|\le C_p|z|^{p-2}$ for every $z\in\R^N\setminus\set*{0}$.
Recalling that $p-2<0$ and that $|a|>2|a-b|$ by our previous assumption, we infer that $|a|^{p-2}\le C_p|a-b|^{p-2}$ and so the desired inequality~\eqref{eq:ele2} follows, concluding the proof.
\end{proof}


\end{document}